\documentclass[11pt]{amsart}
\usepackage{amsmath, amssymb, amsthm}
\usepackage{geometry}
\usepackage{hyperref}
\usepackage{enumitem}
\usepackage{lmodern}
\usepackage{microtype}
\usepackage{mathtools}
\usepackage{cleveref}
\usepackage{lineno}

\theoremstyle{definition}
\newtheorem{definition}{Definition}[section]
\newtheorem{example}[definition]{Example}
\newtheorem{remark}[definition]{Remark}
\newtheorem{assumption}[definition]{Assumption}

\theoremstyle{plain}
\newtheorem{theorem}[definition]{Theorem}
\newtheorem{proposition}[definition]{Proposition}
\newtheorem{lemma}[definition]{Lemma}
\newtheorem{corollary}[definition]{Corollary}

\newcommand{\E}{\mathbb{E}}

\newcommand{\argmin}{\operatorname*{argmin}}

\newcommand{\abs}[1]{|#1|}

\title{Stability of Two-Stage Stochastic Programs under Problem-Dependent Costs}

\author{Nils Peyrouset}
\address{Dipartimento di Informatica\\Università di Pisa\\Largo B. Pontecorvo 3, 56127 Pisa, Italy}
\email{nils.peyrouset@ensae.fr}

\author{Benoît Tran}
\address{Dipartimento di Informatica\\Università di Pisa\\Largo B. Pontecorvo 3, 56127 Pisa, Italy}
\email{benoit.tran@tutanota.com}

\date{\today}


\begin{document}

% \linenumbers
\begin{abstract}
Classical stability theory for stochastic programming relies on the Wasserstein-Fortet-Mourier duality, which requires the ground cost to be a distance. When using problem-dependent costs instead of metrics, this duality no longer yields Fortet-Mourier bounds. This paper develops a direct stability approach using the primal optimal transport formulation. We prove that under minimal regularity conditions and a regret domination property, the optimal value function remains Lipschitz continuous with respect to problem-dependent transport costs. Our approach works directly with transport couplings rather than relying on dual representations to establish stability bounds. We present two applications: (1) For linear programs with continuous second-stage, we show that regret domination holds with constants depending on dual bounds and Lipschitz properties, using sensitivity analysis. (2) For mixed-integer second-stage problems, we show that combinatorial structure can be exploited to obtain tight regret bounds. We analyze several examples as illustrations. These results provide theoretical justification for problem-dependent scenario reduction approaches and enable their application to both continuous and discrete stochastic programs.
\end{abstract}

\maketitle

\setcounter{tocdepth}{1}
\tableofcontents

\section{Introduction}

Two-stage stochastic programming provides a fundamental framework for decision-making under uncertainty, with applications spanning operations research, finance, and engineering \cite{birge2011}. The generic formulation seeks to minimize
\begin{equation}\label{eq:two-stage}
\min_{x \in X} \left\{ g(x) + \mathbb{E}_P[Q(x, \xi)] \right\},
\end{equation}
where $g: \mathbb{R}^{n_1} \to \mathbb{R}$ is the first-stage cost function, $x \in \mathbb{R}^{n_1}$ represents first-stage decisions made before uncertainty is realized, $\xi \in \Xi \subseteq \mathbb{R}^d$ denotes the random parameters with distribution $P$, and $Q(x, \xi)$ captures the optimal recourse cost:
\begin{equation}\label{eq:recourse}
Q(x, \xi) = \min_{z \in Z(x, \xi)} c(z; x, \xi),
\end{equation}
where $Z(x, \xi)$ denotes the feasible set of second-stage decisions given first-stage decision $x$ and scenario realization $\xi$.

To illustrate why problem structure matters for scenario reduction, consider a simple inventory problem where $x$ is the order quantity and $\xi$ is the uncertain demand. For scenarios $\xi_1 = 100$ and $\xi_2 = 200$ with symmetric costs, the Euclidean distance treats these scenarios as equally different from a third scenario $\xi_3 = 150$. However, if the optimal order quantity for $\xi_3$ is $x^*(150) = 140$, then using this decision for $\xi_1$ incurs only a small holding cost, while using it for $\xi_2$ incurs a large shortage cost. A problem-dependent cost that captures this asymmetric regret would recognize that $\xi_3$ is a better representative for $\xi_1$ than for $\xi_2$, leading to more effective scenario reduction.

A central challenge in stochastic programming is stability analysis: quantifying how perturbations in the probability distribution $P$ affect the optimal value and solutions. Stability results are the cornerstone of scenario reduction methods in stochastic programming. When solving real-world stochastic programs, the true probability distribution often involves continuous random variables or an intractably large number of scenarios. Scenario reduction techniques \cite{dupacova2003,heitsch2006,rujeerapaiboon2022} approximate the original distribution $P$ by a simpler distribution $Q$ supported on fewer scenarios, making the problem computationally tractable. The quality of this approximation—and hence the reliability of the computed solutions—depends crucially on stability bounds that quantify how the optimal value $v(P)$ changes when $P$ is replaced by $Q$. Without such bounds, there is no guarantee that solving the reduced problem yields meaningful solutions for the original problem.

Classical stability theory \cite{heitsch2006,rachev2002,romisch2003}, establishes that the optimal value is Lipschitz continuous with respect to Wasserstein distances: $|v(P) - v(Q)| \leq L \cdot W_p(P,Q)$, where $W_p$ is the $p$-Wasserstein distance and $L$ is a problem-dependent Lipschitz constant. This fundamental result enables practitioners to control the approximation error by choosing $Q$ to minimize $W_p(P,Q)$ subject to cardinality constraints. However, these classical results have two fundamental limitations:
\begin{enumerate}
\item They require the ground cost in the optimal transport problem to be a distance (metric), which may fail to capture the specific structure of the optimization problem, potentially leading to conservative bounds.
\item They assume convexity and Lipschitz continuity of the value function $Q(x,\xi)$ in both arguments, assumptions that fail for mixed-integer second-stage problems where the value function is discontinuous and non-convex \cite{birge2011,shapiro2009}.
\end{enumerate}

Recently, there has been growing interest in problem-dependent approaches to scenario reduction. Henrion and R\"omisch \cite{henrion2022} use minimal information distances $d_{\mathcal{F}}$ that still maintains stability bounds. Keutchayan et al. \cite{keutchayan2023} proposed problem-driven scenario clustering that exploits optimization structure, while Bertsimas and Mundru \cite{bertsimas2023} introduced problem-dependent ground costs that measure the \emph{decision regret} rather than scenario distance. Computational experiments in \cite{bertsimas2023,keutchayan2023} demonstrated significant improvements in scenario reduction quality. However, in \cite{bertsimas2023} the theoretical analysis contains a gap: classical stability results which require the ground cost to be a metric cannot be used to problem-dependent costs which are not distances. Indeed, when the standard Euclidean ground metric (or norm-based distance) in the Wasserstein distances is replaced by problem-dependent costs, the dual optimal transport formulation no longer corresponds to Fortet-Mourier metrics, breaking the connection to established stability theory. This leaves several theoretical questions unanswered:
\begin{itemize}
\item How can stability bounds be established when the Wasserstein-Fortet-Mourier duality breaks down?
\item Under what conditions do problem-dependent transport costs yield valid stability bounds?
\item What is the precise relationship between problem-dependent costs and classical stability results?
\end{itemize}

This paper addresses these questions by developing a stability theory that works directly with problem-dependent transport formulations, bypassing the need for Fortet-Mourier duality. Our main result shows that if a symmetric problem-dependent ground cost $c: \Xi \times \Xi \to [0,\infty]$ dominates the regret---meaning that for all scenarios $\xi, \xi' \in \Xi$ and some constant $\beta > 0$, we have 
$$Q(x^*(\xi), \xi') - Q(x^*(\xi'), \xi') \leq \beta \cdot c(\xi,\xi'),$$
 where $x^*(\xi)$ denotes the optimal first-stage decision under scenario $\xi$---then the optimal value function is Lipschitz continuous with respect to the induced transport distance $$|v(P) - v(Q)| \leq \beta \cdot \mathcal{T}_c(P,Q).$$ where $\mathcal{T}_c(P,Q) = \inf_{\pi \in \Pi(P,Q)} \int c(\xi,\xi') d\pi(\xi,\xi')$ is the optimal transport cost between distributions $P$ and $Q$. This establishes stability without requiring $c$ to be a metric, thus providing theoretical justification for problem-dependent scenario reduction approaches.

Our stability result is established through a direct proof that works entirely with transport couplings. While the argument is elementary, it enables stability analysis for problem classes—particularly mixed-integer recourse problems—where the technical machinery of classical theory does not apply.

We extend classical stability results in two directions: from distances to general non-negative lower semicontinuous proper costs, and from continuous problems with Lipschitz value functions to discrete problems where such regularity fails. While we focus on two-stage stochastic programs, we note that similar stability questions arise in multistage settings, where Pflug and Pichler \cite{pflugpichler2012} introduced the nested distance for comparing scenario trees. However, extending our problem-dependent cost framework to multistage programs is beyond the scope of this article. Our contributions are:

\begin{enumerate}[label=(\roman*)]
\item We extend classical stability results to non-negative lower semicontinuous proper costs, developing a direct approach that does not rely on the Wasserstein-Fortet-Mourier duality (Theorem~\ref{thm:main-stability}). This provides a rigorous theoretical foundation for approaches like Bertsimas-Mundru that use general costs rather than distances.
\item We identify minimal regularity conditions ensuring that problem-dependent costs provide stability bounds with explicit constants for both continuous and mixed-integer recourse problems.
\item We demonstrate that our approach enables stability analysis for mixed-integer recourse problems where the value function lacks the convexity and Lipschitz continuity required by classical theory. Through concrete examples including capacitated facility location with single-sourcing and stochastic network design, we show how problem-dependent costs can exploit combinatorial structure to achieve regret bounds, instead of depending on conservative Lipschitz estimates.  
\end{enumerate}

The remainder of this paper is organized as follows. Section~\ref{sec:preliminaries} establishes notation and examines the fundamental relationship between optimal transport duality and Fortet-Mourier metrics, explaining why the Bertsimas-Mundru approach faces theoretical challenges. Section~\ref{sec:ground-costs} introduces problem-dependent ground costs and their properties. Section~\ref{sec:main-results} presents our main stability results and develops the necessary concepts of regret and regret domination. Section~\ref{sec:sufficient-conditions} provides sufficient conditions for establishing regret domination: general conditions for arbitrary programs, specialized results for linear programs using sensitivity analysis, and results for mixed-integer programs exploiting combinatorial structure. Section~\ref{sec:conclusion} concludes with a summary, comparison with classical approaches, and future research directions. 

\section{Preliminaries: Notation and Optimal Transport Duality}\label{sec:preliminaries}

We work in finite-dimensional Euclidean spaces with the Euclidean norm $\|\cdot\|$. For a set $S$, we denote by $\mathcal{P}(S)$ the set of Borel probability measures on $S$. The support of a measure $\mu$ is denoted $\text{supp}(\mu)$.

\subsection{Wasserstein Distances and Probability Metrics}

For $p \geq 1$ and probability measures $\mu, \nu \in \mathcal{P}(\Xi)$, the $p$-Wasserstein distance is
\begin{equation}
W_p(\mu, \nu) = \left( \inf_{\pi \in \Pi(\mu, \nu)} \int_{\Xi \times \Xi} \|\xi - \xi'\|^p \, d\pi(\xi, \xi') \right)^{1/p},
\end{equation}
where $\Pi(\mu, \nu)$ denotes the set of couplings (joint distributions) with marginals $\mu$ and $\nu$.

More generally, for a ground cost $c: \Xi \times \Xi \to [0, +\infty]$, the optimal transport cost is
\begin{equation}
\mathcal{T}_c(\mu, \nu) = \inf_{\pi \in \Pi(\mu, \nu)} \int_{\Xi \times \Xi} c(\xi, \xi') \, d\pi(\xi, \xi').
\end{equation}
We assume throughout that $c$ is measurable and lower semicontinuous. When $c$ is proper and satisfies appropriate growth conditions, the infimum is attained.

\subsection{Fortet-Mourier Metrics}

Following the framework of Rachev and R\"omisch \cite{rachev2002}, we introduce the $p$-th order Fortet-Mourier metric $\zeta_p$ ($p \geq 1$):
\begin{equation}
\zeta_p(\mu, \nu) = \sup_{f \in \mathcal{F}_p(\Xi)} \left| \int_{\Xi} f(\xi) (\mu - \nu)(d\xi) \right|,
\end{equation}
where
\begin{equation}
\mathcal{F}_p(\Xi) = \left\{ f: \Xi \to \mathbb{R} : |f(\xi) - f(\xi')| \leq \max\{1, \|\xi\|^{p-1}, \|\xi'\|^{p-1}\} \|\xi - \xi'\| \right\}.
\end{equation}

The Fortet-Mourier metrics provide a bridge between Wasserstein distances and problem-specific metrics, as they metrize weak convergence plus convergence of $p$-th moments.

\subsection{Function Spaces and Regularity}

We denote by $\text{Lip}_L(\Xi)$ the space of Lipschitz functions on $\Xi$ with Lipschitz constant at most $L$. A function $f: \mathbb{R}^n \times \Xi \to \mathbb{R}$ is called jointly Lipschitz if there exists $L > 0$ such that
\begin{equation}
|f(x, \xi) - f(x', \xi')| \leq L(\|x - x'\| + \|\xi - \xi'\|)
\end{equation}
for all $(x, \xi), (x', \xi') \in \mathbb{R}^n \times \Xi$.

\subsection{Stochastic Programming Notation}

Throughout, we consider the two-stage problem \eqref{eq:two-stage} with the following standing assumptions:
\begin{itemize}
\item The first-stage constraint set $X \subseteq \mathbb{R}^{n_1}$ is closed and nonempty.
\item The first-stage cost $g: \mathbb{R}^{n_1} \to \mathbb{R} \cup \{+\infty\}$ is proper, lower semicontinuous, and coercive.
\item For each $(x, \xi) \in X \times \Xi$, the second-stage constraint set $Z(x, \xi)$ is nonempty (relatively complete recourse).
\item The second-stage cost function $c: \mathbb{R}^{n_2} \times \mathbb{R}^{n_1} \times \Xi \to \mathbb{R}$ is measurable in $\xi$ and lower semicontinuous in $(z, x)$.
\item Under these conditions, optimal solutions exist and the solution sets 
\[
S(P) = \argmin_{x \in X} \{g(x) + \mathbb{E}_P[Q(x, \xi)]\}
\]
are nonempty for any probability measure $P$ with finite expectation.
\end{itemize}

For linear second-stage problems of the form 
\[ 
    Q(x,\xi) = \min\{q(\xi)^T y : W(\xi)y = h(\xi) - T(\xi)x, y \geq 0\},
\]

\begin{itemize}
\item \emph{fixed recourse}: The recourse matrix $W$ is deterministic (not dependent on $\xi$),
\item \emph{complete recourse}: The system $Wy = \chi, y \geq 0$ has a solution for every $\chi \in \mathbb{R}^m$. Equivalently, $\text{pos } W = \{Wy : y \geq 0\} = \mathbb{R}^m$,
\item \emph{relatively complete recourse}: For every $x \in X$ and almost every $\xi \in \Xi$, the second-stage problem is feasible, i.e., $Q(x,\xi) < +\infty$ w.p.1.
\end{itemize}

\subsection{Optimal Transport Duality and Fortet-Mourier Metrics}

Understanding why problem-dependent ground costs break established stability theory requires examining the relationship between optimal transport duality and Fortet-Mourier metrics. This subsection provides background for our results.

\subsubsection{General Kantorovich Duality}

For a general cost function $c: \Xi \times \Xi \to [0, +\infty]$ that is measurable and lower semicontinuous, the optimal transport problem between probability measures $\mu$ and $\nu$ is:
\begin{equation}
\mathcal{T}_c(\mu, \nu) = \inf_{\pi \in \Pi(\mu, \nu)} \int_{\Xi \times \Xi} c(\xi, \xi') \, d\pi(\xi, \xi').
\end{equation}

By Kantorovich duality, when $c$ is lower semicontinuous and $\Xi$ is a Polish space, this equals:
\begin{equation}
\mathcal{T}_c(\mu, \nu) = \sup_{\phi, \psi} \left\{ \int_{\Xi} \phi(\xi) \, d\mu(\xi) + \int_{\Xi} \psi(\xi') \, d\nu(\xi') \right\}
\end{equation}
where the supremum is over functions $\phi, \psi$ satisfying $\phi(\xi) + \psi(\xi') \leq c(\xi, \xi')$ for all $\xi, \xi'$.

\emph{Observation}: The dual constraint $\phi(\xi) + \psi(\xi') \leq c(\xi, \xi')$ determines what types of functions appear in the dual formulation. When $c$ is a distance, this leads to Lipschitz constraints and Fortet-Mourier metrics. When $c$ is problem-dependent, this connection breaks down.

\subsubsection{The Distance Case: Wasserstein-Fortet-Mourier Correspondence}

When the cost function is a distance $c(\xi, \xi') = d(\xi, \xi')$, the Kantorovich dual has a special structure that connects directly to Fortet-Mourier metrics.

\begin{theorem}[1-Wasserstein equals 1-Fortet-Mourier]
\label{thm:w1-fm1}
For probability measures $\mu, \nu$ on a separable metric space $(\Xi, d)$,
\begin{equation}
W_1(\mu, \nu) = \zeta_1(\mu, \nu),
\end{equation}
where $W_1$ is the 1-Wasserstein distance and $\zeta_1$ is the first-order Fortet-Mourier metric.
\end{theorem}

\begin{proof}

The 1-Wasserstein distance is defined by:
\begin{equation}
W_1(\mu, \nu) = \inf_{\pi \in \Pi(\mu, \nu)} \int_{\Xi \times \Xi} d(\xi, \xi') \, d\pi(\xi, \xi'),
\end{equation}
where $\Pi(\mu, \nu)$ is the set of couplings with marginals $\mu$ and $\nu$.

By the Kantorovich-Rubinstein theorem:
\begin{equation}
W_1(\mu, \nu) = \sup_{f \in \text{Lip}_1(\Xi)} \left| \int_{\Xi} f(\xi) (\mu - \nu)(d\xi) \right|,
\end{equation}
where $\text{Lip}_1(\Xi) = \{f: \Xi \to \mathbb{R} : |f(\xi) - f(\xi')| \leq d(\xi, \xi') \text{ for all } \xi, \xi'\}$.

The first-order Fortet-Mourier metric is:
\begin{equation}
\zeta_1(\mu, \nu) = \sup_{f \in \mathcal{F}_1(\Xi)} \left| \int_{\Xi} f(\xi) (\mu - \nu)(d\xi) \right|,
\end{equation}

where $\mathcal{F}_1(\Xi) = \{f: \Xi \to \mathbb{R} : |f(\xi) - f(\xi')| \leq 1 \cdot \|\xi - \xi'\| \}$.
Notice that $\mathcal{F}_1(\Xi)$ is exactly $\text{Lip}_1(\Xi)$. The equality follows immediately. 
\end{proof}

% This result shows the connection between optimal transport duality and classical stability theory.

\begin{lemma}[Wasserstein Distance Ordering]
\label{lem:wasserstein-ordering}
For probability measures $\mu, \nu$ on a metric space $(\Xi, d)$ and $1 \leq p < q$, we have
\begin{equation}
W_p(\mu, \nu) \leq W_q(\mu, \nu).
\end{equation}
\end{lemma}

\begin{proof}
Let $\pi^*$ be an optimal coupling for $W_q(\mu, \nu)$. We use Jensen's inequality: since $t \mapsto t^{p/q}$ is concave on $[0, \infty)$ when $p < q$:
\begin{align}
W_p(\mu, \nu)^p &= \inf_{\pi \in \Pi(\mu, \nu)} \int_{\Xi \times \Xi} d(\xi, \xi')^p \, d\pi(\xi, \xi') \\
&\leq \int_{\Xi \times \Xi} d(\xi, \xi')^p \, d\pi^*(\xi, \xi') \\
&= \int_{\Xi \times \Xi} \left(d(\xi, \xi')^q\right)^{p/q} \, d\pi^*(\xi, \xi') \\
&\leq \left(\int_{\Xi \times \Xi} d(\xi, \xi')^q \, d\pi^*(\xi, \xi')\right)^{p/q} \quad \text{(Jensen's inequality)}\\
&= W_q(\mu, \nu)^p.
\end{align}
Taking $p$-th roots gives $W_p(\mu, \nu) \leq W_q(\mu, \nu)$.
\end{proof}

\begin{corollary}[Classical Stability Bounds via Wasserstein Distances]
\label{cor:classical-bounds}
Consider a two-stage stochastic program satisfying the conditions of Rachev-R\"omisch \cite{rachev2002} Theorem 3.3 (relatively complete recourse, dual feasibility, finite second moments). Then there exists a constant $L > 0$ such that for any $p \geq 1$:
\begin{equation}
|v(P) - v(\nu)| \leq L \cdot W_p(P, \nu),
\end{equation}
where $v(P)$ and $v(\nu)$ are the optimal values under distributions $P$ and $\nu$ respectively.
\end{corollary}

\begin{proof}
By classical stability theory for two-stage stochastic programs (e.g., Rachev-R\"omisch \cite{rachev2002}), there exists a constant $L > 0$ such that $|v(P) - v(\nu)| \leq L \cdot \zeta_1(P, \nu)$ where $\zeta_1$ is the first-order Fortet-Mourier metric. By Theorem~\ref{thm:w1-fm1}, we have $\zeta_1(P, \nu) = W_1(P, \nu)$. Finally, by Lemma~\ref{lem:wasserstein-ordering}, $W_1(P, \nu) \leq W_p(P, \nu)$ for any $p \geq 1$. Combining these inequalities gives the result.
\end{proof}

This corollary shows that whenever classical stability theory applies (through Fortet-Mourier metrics), we obtain stability bounds in terms of any $p$-Wasserstein distance with $p \geq 1$. The proof demonstrates the clear chain: optimal value differences are bounded by the first-order Fortet-Mourier metric, which equals the 1-Wasserstein distance, which is in turn dominated by higher-order Wasserstein distances. However, this approach requires the Wasserstein-Fortet-Mourier correspondence, which fails for problem-dependent ground costs. The direct approach (Theorem~\ref{thm:main-stability}) provides stability bounds for problem-dependent transport without relying on this classical correspondence.

\subsubsection{The Problem-Dependent Case: Breaking the Duality}

When the cost function $c(\xi, \xi')$ is problem-dependent (such as the Bertsimas-Mundru divergence), the Kantorovich dual formulation no longer corresponds to any Fortet-Mourier metric.

For a lower semicontinuous ground cost $c: \Xi \times \Xi \to [0, +\infty]$, the Kantorovich dual formulation is:
\begin{equation}
\mathcal{T}_c(P, \nu) = \sup_{(\phi, \psi) \in \Phi_c} \left\{ \int_\Xi \phi \, dP + \int_\Xi \psi \, d\nu \right\},
\end{equation}
where $\Phi_c = \{(\phi, \psi) \in C_b(\Xi) \times C_b(\Xi) : \phi(\xi) + \psi(\xi') \leq c(\xi, \xi') \text{ for all } \xi, \xi' \in \Xi\}$ and $C_b(\Xi)$ denotes the space of continuous bounded functions on $\Xi$.

Consider the Bertsimas-Mundru cost $c(\xi, \xi') = c_{BM}(\xi, \xi') = Q(x^*(\xi'), \xi) - Q(x^*(\xi), \xi) + \alpha \|\xi - \xi'\|$. The Kantorovich dual constraint becomes:
\begin{equation}
\phi(\xi) + \psi(\xi') \leq Q(x^*(\xi'), \xi) - Q(x^*(\xi), \xi) + \alpha \|\xi - \xi'\|.
\end{equation}

\emph{Key observation}: This constraint depends on the optimization problem structure through $Q$ and $x^*$, not just on the geometry of the space $\Xi$. Therefore:

\begin{enumerate}
\item The admissible functions $\phi, \psi$ in the dual are determined by optimization properties, not Lipschitz conditions.
\item The resulting dual formulation cannot be expressed as a Fortet-Mourier metric of the form $\sup_{f \in \mathcal{F}} |\int f \, d(\mu - \nu)|$.
\item Classical stability theory, which relies on the Wasserstein-Fortet-Mourier correspondence, no longer applies.
\end{enumerate}

This explains why the Bertsimas-Mundru approach requires additional theoretical development: while their computational results are excellent, the lack of connection to Fortet-Mourier metrics means that classical stability theory does not directly apply. Theorem~\ref{thm:main-stability} provides the necessary stability bounds through a direct approach that bypasses the need for Fortet-Mourier duality.

\section{Problem-Dependent Ground Costs}\label{sec:ground-costs}

\subsection{General Framework}

Following the method of probability metrics introduced by Rachev and R\"omisch \cite{rachev2002}, we develop problem-dependent ground costs that capture the structure of stochastic programs. The observation is that the stability of stochastic programs depends on the behavior of the objective function, not just the distance between scenarios.

\begin{definition}[Problem-Dependent Ground Cost]
\label{def:pd-cost}
Given a two-stage stochastic program with second-stage value function $Q(x, \xi)$, a \emph{problem-dependent ground cost} is a function $c: \Xi \times \Xi \to [0, +\infty]$ that measures the dissimilarity between scenarios from the perspective of the optimization problem. We require:
\begin{enumerate}[label=(\alph*)]
\item \emph{Measurability}: $c$ is measurable with respect to the product $\sigma$-algebra on $\Xi \times \Xi$.
\item \emph{Lower semicontinuity}: $c$ is lower semicontinuous.
\item \emph{Properness}: $c$ is proper, i.e., not identically infinite. Formally, there exists at least one pair $(\xi, \xi') \in \Xi \times \Xi$ such that $c(\xi, \xi') < +\infty$.
\end{enumerate}
These conditions ensure that the optimal transport problem $\mathcal{T}_c(\mu, \nu)$ is well-defined. For asymmetric costs, we define the symmetrized version: $c_S(\xi, \xi') = \frac{1}{2}[c(\xi, \xi') + c(\xi', \xi)]$.
\end{definition}

\subsection{Examples of Problem-Dependent Costs}

\begin{example}[Bertsimas-Mundru Cost]
The Bertsimas-Mundru problem-dependent ground cost \cite{bertsimas2023} measures the suboptimality of using the wrong decision. The complete formulation includes a regularization term:
\begin{equation}
c_{BM}(\xi, \xi') = Q(x^*(\xi'), \xi) - Q(x^*(\xi), \xi) + \alpha \|\xi - \xi'\|,
\end{equation}
where $x^*(\xi) \in \argmin_{x \in X} \{g(x) + Q(x, \xi)\}$ is the optimal first-stage decision when the scenario is $\xi$, and $\alpha > 0$ is a regularization parameter.

This cost combines two components:
\begin{itemize}
\item the decision regret $Q(x^*(\xi'), \xi) - Q(x^*(\xi), \xi)$, measuring the suboptimality of using the wrong decision,
\item the regularization term: $\alpha \|\xi - \xi'\|$, which ensures the cost is positive when $\xi \neq \xi'$ and provides Lipschitz continuity.
\end{itemize}

For the symmetric version used in their optimal transport formulation:
\begin{equation}
c_S(\xi, \xi') = \frac{1}{2}\left[c_{BM}(\xi, \xi') + c_{BM}(\xi', \xi)\right]
\end{equation}

Expanding the definition:
\begin{align}
c_S(\xi, \xi') = \frac{1}{2}\bigg[&\left(Q(x^*(\xi'), \xi) - Q(x^*(\xi), \xi) + \alpha \|\xi - \xi'\|\right) \\
&+ \left(Q(x^*(\xi), \xi') - Q(x^*(\xi'), \xi') + \alpha \|\xi' - \xi\|\right)\bigg].
\end{align}

Since $\|\xi - \xi'\| = \|\xi' - \xi\|$, this simplifies to:
\begin{equation}
c_S(\xi, \xi') = \frac{1}{2}\left[Q(x^*(\xi'), \xi) - Q(x^*(\xi), \xi) + Q(x^*(\xi), \xi') - Q(x^*(\xi'), \xi')\right] + \alpha \|\xi - \xi'\|.
\end{equation}

\emph{Key Properties}:
\begin{itemize}
\item $c_{BM}(\xi, \xi') \geq \alpha \|\xi - \xi'\| \geq 0$ (non-negativity).
\item $c_{BM}(\xi, \xi) = 0$ and $c_{BM}(\xi, \xi') = 0 \Rightarrow \xi = \xi'$.
\item Generally asymmetric: $c_{BM}(\xi, \xi') \neq c_{BM}(\xi', \xi)$ unless 
\[ 
Q(x^*(\xi'), \xi) - Q(x^*(\xi), \xi) = Q(x^*(\xi), \xi') - Q(x^*(\xi'), \xi').
\]
\end{itemize}
The triangle inequality may not be satisfied by $c_{BM}$ or its symmetrized counterpart $c_S$, so they are not distances. Thus, classical stability results relying on the Wasserstein-Fortet-Mourier duality breaks down for $c_{BM}$ or $c_S$. 
\end{example}

\begin{example}[Inventory Management]
Consider a newsvendor problem where $x \in \mathbb{R}_+$ represents the order quantity and $\xi \in \mathbb{R}_+$ is the random demand:
\begin{align}
g(x) &= cx \quad \text{(ordering cost)}, \\
Q(x, \xi) &= h \max\{x - \xi, 0\} + p \max\{\xi - x, 0\} \quad \text{(holding + penalty costs)}.
\end{align}

For scenarios $\xi_1 = 10$ and $\xi_2 = 20$, suppose the optimal decisions are $x^*(10) = 12$ and $x^*(20) = 18$. Then:
\begin{align}
c_{BM}(10, 20) &= Q(18, 10) - Q(12, 10) \\
&= [h \cdot 8 + p \cdot 0] - [h \cdot 2 + p \cdot 0] \\
&= 6h.
\end{align}

This captures the extra holding cost incurred by over-ordering when demand is low. This example illustrates how $c_{BM}$ captures the economic cost of suboptimal decisions rather than mere scenario dissimilarity.
\end{example}

\begin{example}[Average Regret Cost]
For a finite set of representative decisions $\{x_1, \ldots, x_K\} \subset X$:
\begin{equation}
c_{avg}(\xi, \xi') = \frac{1}{K} \sum_{k=1}^K |Q(x_k, \xi) - Q(x_k, \xi')|
\end{equation}
This measures the average change in objective value across representative decisions. The cost satisfies Definition~\ref{def:pd-cost} under standard continuity assumptions.
\end{example}

\begin{example}[Composite Problem-Dependent Cost]
\label{ex:composite-cost}
The framework enables combining multiple problem-dependent components to exploit different aspects of the problem structure. Consider the composite cost:
\begin{equation}
c_{comp}(\xi, \xi') = \alpha \|\xi - \xi'\| + \beta \|x^*(\xi) - x^*(\xi')\| + \gamma c_{BM}(\xi, \xi'),
\end{equation}
where:
\begin{itemize}
\item $\alpha \|\xi - \xi'\|$: baseline scenario distance (ensures metric properties),
\item $\beta \|x^*(\xi) - x^*(\xi')\|$: decision stability component,
\item $\gamma c_{BM}(\xi, \xi')$: Bertsimas-Mundru regret component,
\end{itemize}
with weights $\alpha, \beta, \gamma \geq 0$ and $\alpha > 0$.

This composite cost satisfies Definition~\ref{def:pd-cost} and allows practitioners to:
\begin{enumerate}
\item weight different aspects of problem structure according to their importance,
\item incorporate multiple sources of information (noise space, decision space, value space),
\item adapt to specific problem characteristics by tuning the weights.
\end{enumerate}
\end{example}

% The composite cost example shows how our approach enables flexible ground cost design.

\section{Main Stability Result: Regret Domination implies Stability}\label{sec:main-results}

We now establish stability results for two-stage stochastic programs using problem-dependent costs. We introduce the concept of \emph{regret} and \emph{regret domination} which relates problem-dependent costs to the worst-case behavior of the value function.

\subsection{Regret and Regret Domination}

\begin{definition}[Regret]
The regret from scenario $\xi'$ to scenario $\xi$ is defined as:
\begin{equation}
R(\xi, \xi') := \sup_{x \in X} [Q(x, \xi) - Q(x, \xi')].
\end{equation}
This measures the worst-case increase in second-stage cost when the scenario changes from $\xi'$ to $\xi$.
\end{definition}

The key property we need is that the problem-dependent cost controls the worst-case regret:

\begin{assumption}[Regret Domination]
\label{ass:regret-dom}
Let $c: \Xi \times \Xi \to [0, +\infty]$ be a problem-dependent ground cost. There exists a constant $\beta > 0$ such that for all $\xi, \xi' \in \Xi$:
\begin{equation}
R(\xi, \xi') \leq \beta \cdot c(\xi, \xi').
\end{equation}
\end{assumption}

Regret domination means that the regret over all decisions is controlled by the chosen problem-dependent ground cost. This is reasonable when the cost $c$ captures the structure of how the value function varies with scenarios. For the Bertsimas-Mundru cost, this means that the regret between optimal decisions bounds the worst-case regret. Note that both $R$ and $c$ may be asymmetric.

\subsection{Main Stability Theorem}

The key insight is that stability can be achieved directly through the transport formulation without requiring the classical Wasserstein-Fortet-Mourier duality.

\begin{theorem}[Direct Stability for Problem-Dependent Transport]
\label{thm:main-stability}
Consider two-stage stochastic programs:
\begin{align}
v(P) &= \min_{x \in X} \left\{ g(x) + \E_P[Q(x, \xi)] \right\}, \\
v(\nu) &= \min_{x \in X} \left\{ g(x) + \E_\nu[Q(x, \xi)] \right\}.
\end{align}

We assume
\begin{enumerate}[label=(\roman*)]
\item \emph{well-posed two-stage program}: for each $P \in \mathcal{P}(\Xi)$, the problem $\min_{x \in X} \{g(x) + \mathbb{E}_P[Q(x, \xi)]\}$ has an optimal solution,
\item \emph{integrability}: for each $x \in X$, the function $\xi \mapsto Q(x, \xi)$ is measurable and $\mathbb{E}_P[|Q(x, \xi)|] < \infty$ for all relevant probability distributions $P$,
\item \emph{regret domination}: there exist a function $c: \Xi \times \Xi \to [0, +\infty]$ and a constant $\beta > 0$ such that 
\begin{equation*}
R(\xi, \xi') \leq \beta \cdot c(\xi, \xi') \quad \text{for all } \xi, \xi' \in \Xi,
\end{equation*}
where $R(\xi, \xi') = \sup_{x \in X} [Q(x, \xi) - Q(x, \xi')]$.
\end{enumerate}

Then, we have that
\begin{equation}
v(P) - v(\nu) \leq \beta \cdot \mathcal{T}_c(P, \nu) \quad \text{and} \quad v(\nu) - v(P) \leq \beta \cdot \mathcal{T}_c(\nu, P).
\end{equation}

Therefore, we have that
\begin{equation}
\abs{v(P) - v(\nu)} \leq \beta \cdot \max\{\mathcal{T}_c(P, \nu), \mathcal{T}_c(\nu, P)\}.
\end{equation}
\end{theorem}

\begin{proof}
Let $x_\nu^* \in \argmin_{x \in X} \{g(x) + \mathbb{E}_\nu[Q(x, \xi)]\}$ be optimal for distribution $\nu$. By optimality of $x_\nu^*$, we have
\begin{align}
v(P) - v(\nu) &\leq [g(x_\nu^*) + \mathbb{E}_P[Q(x_\nu^*, \xi)]] - [g(x_\nu^*) + \mathbb{E}_\nu[Q(x_\nu^*, \xi)]] \\
&= \mathbb{E}_P[Q(x_\nu^*, \xi)] - \mathbb{E}_\nu[Q(x_\nu^*, \xi)].
\end{align}

For any coupling $\pi \in \Pi(P, \nu)$:
\begin{align}
\mathbb{E}_P[Q(x_\nu^*, \xi)] - \mathbb{E}_\nu[Q(x_\nu^*, \xi)] &= \int_{\Xi \times \Xi} [Q(x_\nu^*, \xi) - Q(x_\nu^*, \xi')] \, d\pi(\xi, \xi') \\
&\leq \int_{\Xi \times \Xi} R(\xi, \xi') \, d\pi(\xi, \xi') \\
&\leq \beta \int_{\Xi \times \Xi} c(\xi, \xi') \, d\pi(\xi, \xi'),
\end{align}
where the last inequality uses regret domination.

Taking the infimum over all couplings yields $v(P) - v(\nu) \leq \beta \cdot \mathcal{T}_c(P, \nu)$. By symmetry, $v(\nu) - v(P) \leq \beta \cdot \mathcal{T}_c(\nu, P)$.
\end{proof}

% Standard conditions from stochastic programming ensure well-posedness (see text for details).

\begin{corollary}[Stability with Symmetric Costs]
\label{cor:symmetric-costs}
Under the same assumptions as Theorem~\ref{thm:main-stability}, if the ground cost $c$ is symmetric (i.e., $c(\xi, \xi') = c(\xi', \xi)$ for all $\xi, \xi' \in \Xi$), then:
\begin{equation}
|v(P) - v(\nu)| \leq \beta \cdot \mathcal{T}_c(P, \nu).
\end{equation}
\end{corollary}

\begin{proof}
For symmetric costs, $\mathcal{T}_c(P, \nu) = \mathcal{T}_c(\nu, P)$ since the optimal transport problem is invariant under swapping the roles of $P$ and $\nu$ when the cost is symmetric. Therefore, we have that
\begin{equation}
\max\{\mathcal{T}_c(P, \nu), \mathcal{T}_c(\nu, P)\} = \mathcal{T}_c(P, \nu).
\end{equation}
The result follows immediately from Theorem~\ref{thm:main-stability}.
\end{proof}

\section{Sufficient Conditions for Regret Domination}\label{sec:sufficient-conditions}

While Section~\ref{sec:main-results} established that regret domination by a given problem-dependent ground cost implies stability by the associated transport cost. This section provides sufficient conditions under which regret domination holds. We present a hierarchy of results: the classical case of linear programs with costs derived from a norm, problem dependant bound for linear programs exploiting duality and results for mixed-integer programs exploiting either their LP relaxation or their combinatorial structure. We also give a stability result for the problem dependent ground cost proposed in \cite{bertsimas2023}.

\subsection{General Condition: Regret Domination from Classical Lipschitz Bounds}

A fundamental result in stochastic programming stability theory establishes Lipschitz properties of the value function. We recall without proof the following result from R\"omisch \cite{romisch2003}.

\begin{proposition}[R\"omisch \cite{romisch2003}, Proposition 22]
\label{prop:romisch-22}
Consider a two-stage stochastic linear program with fixed recourse:
\begin{equation}
Q(x, \xi) = \min\{q(\xi)^T y : Wy = h(\xi) - T(\xi)x, y \geq 0\}.
\end{equation}

Assume that:
\begin{itemize}
\item[(A1)] For each $(x,\xi) \in X \times \Xi$, the system is feasible (i.e., $h(\xi) - T(\xi)x \in \text{pos } W$) and the dual feasible region is nonempty.
\item[(A2)] The uncertainty set $\Xi$ is a polyhedral subset of $\mathbb{R}^s$.
\item[(A3)] The functions $q(\cdot)$, $h(\cdot)$, and $T(\cdot)$ depend affine linearly on $\xi$.
\end{itemize}

Then the value function $Q$ satisfies: there exist constants $L > 0$, $r > 0$ such that for all $\xi, \xi' \in \Xi$ and $x \in X$ with $\|x\| \leq r$:
\begin{align}
|Q(x, \xi) - Q(x, \xi')| &\leq L \cdot r \cdot \max\{1, \|\xi\|, \|\xi'\|\} \cdot \|\xi - \xi'\|, \label{eq:romisch-lip-xi}.
\end{align}
\end{proposition}

Under the conditions of the above Proposition, if the set $X$ is bounded by $r$, then the regret satisfies
\begin{equation}
R(\xi, \xi') = \sup_{x \in X} [Q(x, \xi) - Q(x, \xi')] \leq L \cdot r \cdot \max\{1, \|\xi\|, \|\xi'\|\} \cdot \|\xi - \xi'\| \leq L \cdot r \cdot \lVert \xi - \xi' \rVert.
\end{equation}
Let $c(\xi,\xi') = \lVert \xi - \xi' \rVert$ be the usual distance derived from a given norm $\lVert \cdot \rVert$ on the euclidean space $\Xi$. From Corollary~\ref{cor:symmetric-costs}, we have stability of the value function, that is,
\begin{equation}
    |v(P) - v(\nu)| \leq L \cdot r \cdot \mathcal{T}_c(P, \nu) = L \cdot r \cdot W_1(P, \nu).
\end{equation}

We have shown that our stability result Corollary~\ref{cor:symmetric-costs} can also be applied to the classical case of distances that derive from a norm on the noise space $\Xi$. One of the drawback of such approach is that in the context of scenario reduction, the stability bound may be loose as it is independant of the problem's data.

\subsection{Two-stage Stochastic Programs with Continuous Second Stage}

Consider a two-stage stochastic program where the second-stage problem is a linear program with fixed recourse and cost:
\begin{equation}
    \label{prob:lp-sensitivity}
Q(x, \xi) = \min\{q^T z : W z = h(\xi) - T(\xi)x, z \geq 0\}.
\end{equation}

Define the linear program sensitivity cost
\begin{equation}
c_{LP}(\xi, \xi') := M_{\pi} \cdot \left[ \|h(\xi) - h(\xi')\| + r \|T(\xi) - T(\xi')\| \right],
\end{equation}
where $M_{\pi} = \sup\{\|\pi\| : W^T\pi \leq q\}$ is the uniform bound on the dual feasible set.

% \textbf{TODO: precise the regularity needed and h and T to ensure that $c_{LP}$ is a problem-dependent ground cost. It should be much broader than h and T linear. I suppose that absolutely no assumption are needed if the noises are finite and loose stuff if noises are not finite.}
\begin{assumption}
Assume that $c_{LP}$ is a problem-dependent ground cost as in Definition~\ref{def:pd-cost}.
\end{assumption}

For the ground cost $c_{LP}$ to be a valid problem-dependent ground cost in the sense of Definition~\ref{def:pd-cost}, we need minimal regularity on $h$ and $T$. The following result extends classical stability results of R\"omisch and co-authors (see \cite{romisch2003} for instance) which typically assume that $h$ and $T$ are affine.  

\begin{proposition}
\label{prop:lp-sensitivity}

Assume the following on the data of Problem~\eqref{prob:lp-sensitivity},
\begin{itemize}
\item \emph{relatively complete recourse}: for all $x \in X$ and $\xi \in \Xi$, the second-stage problem is feasible,
\item \emph{bounded dual feasible region}: the dual feasible set $\{\pi : W^T\pi \leq q\}$ is bounded,
\item \emph{bounded first-stage decisions}: we assume that $\sup_{x \in X} \|x\| \leq r$.
\end{itemize}

Then the regret satisfies
\begin{equation}
R(\xi, \xi') := \sup_{x \in X} [Q(x, \xi) - Q(x, \xi')] \leq c_{LP}(\xi, \xi').
\end{equation}
\end{proposition}

\begin{proof}
Let $x \in X$ be fixed. For every $\xi \in \Xi$, strong duality holds for Problem~\eqref{prob:lp-sensitivity} and its dual is
\begin{equation}
\max_{\pi} \{ \pi^T[h(\xi) - T(\xi)x] : W^T\pi \leq q \}.
\end{equation}
Since the dual feasible set $\{\pi : W^T\pi \leq q\}$ is independent of $\xi$ and bounded by assumption, for any optimal dual solution $\pi^*(\xi, x)$, we have $\|\pi^*(\xi, x)\| \leq M_{\pi}$.

Let $\pi^*(\xi, x)$ be an optimal dual solution for scenario $\xi$ and first-stage decision $x$. By strong duality for Problem~\eqref{prob:lp-sensitivity} at $(x,\xi)$, $$Q(x,\xi) = [\pi^*(\xi, x)]^T[h(\xi) - T(\xi)x],$$ and by weak duality (since $\pi^*(\xi, x)$ is dual feasible) we have
\begin{align}
Q(x, \xi') &\geq [\pi^*(\xi, x)]^T[h(\xi') - T(\xi')x].
\end{align}

Therefore:
\begin{align}
Q(x, \xi) - Q(x, \xi') &\leq [\pi^*(\xi, x)]^T[h(\xi) - T(\xi)x] - [\pi^*(\xi, x)]^T[h(\xi') - T(\xi')x]\\
&= [\pi^*(\xi, x)]^T \left([h(\xi) - h(\xi')] - [T(\xi) - T(\xi')]x\right)\\
&\leq \lVert \pi^*(\xi, x) \rVert_{\infty} \cdot \lVert [h(\xi) - h(\xi')] - [T(\xi) - T(\xi')]x \rVert_1  \tag{Hölder} \\
&\leq M_{\pi} \cdot \left( \lVert h(\xi) - h(\xi') \rVert_1 + \lVert [T(\xi) - T(\xi')]x \rVert_1 \right)\\
&\leq M_{\pi} \cdot \left( \lVert h(\xi) - h(\xi') \rVert_1 + \lVert T(\xi) - T(\xi') \rVert_1 \cdot \lVert x \rVert_{\infty} \right)\\
&\leq M_{\pi} \cdot \left( \lVert h(\xi) - h(\xi') \rVert_1 + R \lVert T(\xi) - T(\xi') \rVert_1 \right).
\end{align}

Taking the supremum over $x \in X$ yields $R(\xi, \xi') \leq c_{LP}(\xi, \xi')$.
\end{proof}

Thus, regret domination holds for $c_{LP}$. Then one gets stability by applying Proposition~\ref{prop:lp-sensitivity} to the problem dependent ground cost $c_{LP}$ and get that
\begin{equation}
    \lvert v(P) - v(\nu) \rvert \leq \mathcal{T}_{c_{LP}} (P, \nu).
\end{equation}

\subsection{Bertsimas-Mundru ground cost}
We now show how the Bertsimas-Mundru ground cost can be used for linear programs with fixed recourse. This cost function, which measures decision regret plus a regularization term, provides a concrete example of how problem-dependent costs can be designed and analyzed.

\begin{corollary}[Regret Domination for the Bertsimas-Mundru Cost]
\label{cor:linear-regret-dom}
Consider a two-stage linear stochastic program with fixed recourse and the Bertsimas-Mundru ground cost:
\begin{equation}
c_{BM}(\xi, \xi') = Q(x^*(\xi'), \xi) - Q(x^*(\xi), \xi) + \alpha \|\xi - \xi'\|,
\end{equation}
where $x^*(\xi) \in \argmin_{x \in X} \{g(x) + Q(x, \xi)\}$ is the optimal first-stage decision for scenario $\xi$, and $\alpha > 0$ is a regularization parameter.

Under the assumptions of Proposition~\ref{prop:lp-sensitivity} (relatively complete recourse, bounded dual feasible region, and bounded first-stage decisions), and additionally assuming that $h(\cdot)$ and $T(\cdot)$ are Lipschitz continuous with constants $L_h$ and $L_T$ respectively, the Bertsimas-Mundru cost satisfies regret domination with constant:
\begin{equation}
\beta = \frac{M_{\pi}(L_h + R L_T)}{\alpha}.
\end{equation}

Therefore, by Theorem~\ref{thm:main-stability}, we obtain the stability bound:
\begin{equation}
|v(P) - v(\nu)| \leq \frac{M_{\pi}(L_h + R L_T)}{\alpha} \cdot \mathcal{T}_{c_{BM}}(P, \nu).
\end{equation}
\end{corollary}

\begin{proof}
The Bertsimas-Mundru cost is defined as:
\begin{equation}
c_{BM}(\xi, \xi') = Q(x^*(\xi'), \xi) - Q(x^*(\xi), \xi) + \alpha \|\xi - \xi'\|.
\end{equation}

By construction, $c_{BM}(\xi, \xi') \geq \alpha \|\xi - \xi'\|$ since the first term is non-negative. Therefore, the Bertsimas-Mundru cost satisfies the minimal growth property required in Proposition~\ref{prop:lp-sensitivity} part (3), which immediately gives regret domination with constant $\beta = \frac{M_{\pi}(L_h + R L_T)}{\alpha}$.
\end{proof}

The Bertsimas-Mundru cost has two components with distinct roles:
\begin{enumerate}
\item \emph{Decision regret term}: $Q(x^*(\xi'), \xi) - Q(x^*(\xi), \xi)$ measures the suboptimality of using the wrong first-stage decision. This captures the economic cost of misidentifying the scenario.
\item \emph{Regularization term}: $\alpha \|\xi - \xi'\|$ ensures the minimal growth property required in Proposition~\ref{prop:lp-sensitivity}.
\end{enumerate}

The parameter $\alpha$ controls the trade-off, small $\alpha$ emphasizes decision-specific structure, potentially yielding tighter scenario clusters but weaker stability guarantees (larger $\beta$). While large $\alpha$ provides stronger stability guarantees (smaller $\beta$) but may dilute the problem-specific information.
For linear programs, the Bertsimas-Mundru approach offers several practical advantages:
\begin{enumerate}
\item \emph{Computational tractability}: Both $Q(x^*(\xi'), \xi)$ and $Q(x^*(\xi), \xi)$ can be evaluated by solving linear programs.
\item \emph{Economic interpretation}: The cost directly measures the financial impact of scenario misidentification.
\item \emph{Flexibility}: the regularization parameter $\alpha$ can be tuned based on the desired stability-accuracy trade-off.
\end{enumerate}

However, computing the exact value of $\beta$ requires knowing $M_{\pi}$, which may be challenging in practice. Conservative estimates can be used at the expense of potentially looser bounds.

\begin{example}[Stochastic Unit Commitment]
Consider a power system with generating units $i \in I$ and time periods $t \in T$. The first-stage decisions are binary commitment variables $u_{it} \in \{0,1\}$ indicating whether unit $i$ is online at time $t$. Given these commitments, the second-stage problem under demand scenario $\xi$ is:
\begin{align}
Q(u, \xi) = \min \quad & \sum_{i,t} c_i p_{it}(\xi) + \sum_t \rho s_t(\xi) \\
\text{s.t.} \quad & \sum_i p_{it}(\xi) + s_t(\xi) = D_t(\xi) \quad \forall t \\
& p_i^{\min} u_{it} \leq p_{it}(\xi) \leq p_i^{\max} u_{it} \quad \forall i,t \\
& -r_i^{\text{down}} \leq p_{it}(\xi) - p_{i,t-1}(\xi) \leq r_i^{\text{up}} \quad \forall i,t \\
& p_{it}(\xi), s_t(\xi) \geq 0 \quad \forall i,t
\end{align}
where:
\begin{itemize}
\item $p_{it}(\xi)$ is the power output of unit $i$ at time $t$ (continuous)
\item $s_t(\xi)$ is load shedding at time $t$ (continuous)
\item $D_t(\xi)$ is the uncertain demand at time $t$.
\item $c_i$ is the generation cost and $\rho$ is the load shedding penalty.
\end{itemize}

Since all second-stage variables are continuous, Proposition~\ref{prop:lp-sensitivity} can be applied. The dual variables $\pi_t(\xi)$ represent the marginal value of energy at time $t$ (locational marginal prices), and the regret bound becomes:
\begin{equation}
|Q(u, \xi) - Q(u, \xi')| \leq \sum_t \bar{\pi}_t |D_t(\xi) - D_t(\xi')|,
\end{equation}
where $\bar{\pi}_t = \sup_{\xi, u} |\pi_t(\xi, u)|$ bounds the energy prices.

This suggests the problem-dependent ground cost
\begin{equation}
c_{UC}(\xi, \xi') = \sum_t \bar{\pi}_t |D_t(\xi) - D_t(\xi')|,
\end{equation}
which weights demand differences by their economic impact through energy prices.
\end{example}

The stochastic unit commitment example illustrates an important class of problems where:
\begin{itemize}
\item first-stage decisions may be discrete, here the unit commitments,
\item second-stage decisions are continuous, here the dispatch levels.
\end{itemize}

Even though the first stage variables are the binary decisions $u$, the second-stage problem is a parametric LP in the uncertain demands $D_t(\xi)$ and we can apply Proposition~\ref{prop:lp-sensitivity}. This structure appears in many applications including facility location with continuous allocation, network design with continuous flows...

\subsection{Two-stage Stochastic Programs with Mixed-Integer Second Stage}

For mixed-integer second-stage problems, establishing regret bounds is more challenging than for continuous problems due to the lack of strong duality. We present two approaches: a general result using integrality gap bounds, and problem-specific analyses that exploit combinatorial structure.

\subsubsection{General Bounds via Integrality Gap}

When the uncertainty only affects the right-hand side $h(\xi)$ and technology matrix $T(\xi)$, we can derive bounds using the LP relaxation and integrality gap estimates. The proposition handles uncertainty in $h(\xi)$ and $T(\xi)$ because these only affect the dual objective value, not the dual feasible region.

\begin{proposition}[MILP Regret Bounds via Integrality Gap]
\label{prop:milp-integrality-gap}
Consider a mixed-integer second-stage problem:
\begin{equation}
Q(x, \xi) = \min_{y,z} \{q^T y + r^T z : W_y y + W_z z \leq h(\xi) - T(\xi)x, y \geq 0, z \in \mathbb{Z}^{n_I}_+\},
\end{equation}
where only $h(\xi)$ and $T(\xi)$ depend on the scenario $\xi$.

Let $Q_{LP}(x, \xi)$ denote the LP relaxation value. Assume:
\begin{itemize}
\item The LP relaxation has bounded dual optimal solutions: there exists $M_{\pi} > 0$ such that for all optimal dual solutions $\pi^*(\xi)$ and all scenarios $\xi \in \Xi$, we have $\|\pi^*(\xi)\|_{\infty} \leq M_{\pi}$.
\item The integrality gap is uniformly bounded: there exists $\gamma > 0$ such that $\sup_{x \in X, \xi \in \Xi} [Q(x, \xi) - Q_{LP}(x, \xi)] \leq \gamma$.
\item First-stage decisions are bounded: $\|x\|_{\infty} \leq R$ for all $x \in X$.
\end{itemize}

Then the regret $R(\xi, \xi') = \sup_{x \in X} [Q(x, \xi) - Q(x, \xi')]$ satisfies:
\begin{equation}
R(\xi, \xi') \leq M_{\pi} \left[\|h(\xi) - h(\xi')\|_1 + R\|T(\xi) - T(\xi')\|_1\right] + \gamma.
\end{equation}
\end{proposition}

\begin{proof}
For the LP relaxation, we can apply the same sensitivity analysis as in Proposition~\ref{prop:lp-sensitivity}. The dual feasible region is 
\begin{equation}
\{\pi : W_y^T \pi \leq q, W_z^T \pi \leq r\},
\end{equation}
which is independent of $\xi$ since only $h(\xi)$ and $T(\xi)$ vary with the scenario. That is, any admissible dual multipliers $\pi(\xi)$ of $Q(x,\xi)$ is also admissible for $Q(x,\xi')$. 
% \textbf{todo: put a name to the problem instead of using the value.}

Let $\pi^*(\xi')$ be optimal dual multipliers for the LP relaxation at scenario $\xi'$. As $\pi^*(\xi')$ is dual admissible for $Q_{LP}(x, \xi)$, we have $Q_{LP}(x, \xi) \geq \pi^*(\xi')^T[h(\xi) - T(\xi)x]$. By strong duality, $Q_{LP}(x, \xi') = \pi^*(\xi')^T[h(\xi') - T(\xi')x]$. Thus:
\begin{align}
Q_{LP}(x, \xi) - Q_{LP}(x, \xi') &\geq \pi^*(\xi')^T[h(\xi) - T(\xi)x] - \pi^*(\xi')^T[h(\xi') - T(\xi')x] \\
&= \pi^*(\xi')^T[(h(\xi) - h(\xi')) - (T(\xi) - T(\xi'))x].
\end{align}

Multiplying both sides by $-1$, we get
\begin{equation}
Q_{LP}(x, \xi') - Q_{LP}(x, \xi) \leq \pi^*(\xi')^T[(h(\xi') - h(\xi)) - (T(\xi') - T(\xi))x].
\end{equation}

Similarly, with optimal dual multipliers $\pi^*(\xi)$ for scenario $\xi$ we get
\begin{equation}
Q_{LP}(x, \xi') - Q_{LP}(x, \xi) \geq \pi^*(\xi)^T[(h(\xi') - h(\xi)) - (T(\xi') - T(\xi))x]
\end{equation}

Thus, we have 
\begin{align}
\lvert Q_{LP}(x, \xi') - Q_{LP}(x, \xi) \rvert \leq \max \left\{ \lvert \pi^*(\xi')^T[(h(\xi') - h(\xi)) - (T(\xi') - T(\xi))x \rvert, \right.\\ \left. \lvert \pi^*(\xi)^T[(h(\xi') - h(\xi)) - (T(\xi') - T(\xi))x \rvert  \right\}.
\end{align}

By Hölder's inequality on each term in the r.h.s., with $\|\pi^*(\xi')\|_{\infty} \leq M_{\pi}$ and $\|x\|_{1} \leq \|x\|_{\infty} \leq R$,
\begin{equation}
Q_{LP}(x, \xi') - Q_{LP}(x, \xi) \leq M_{\pi} \left[\|h(\xi') - h(\xi)\|_1 + R\|T(\xi') - T(\xi)\|_1\right].
\end{equation}

Now, using the relationship between the MILP and its LP relaxation, we have that
\begin{align}
Q(x, \xi) - Q(x, \xi') &= [Q(x, \xi) - Q_{LP}(x, \xi)] + [Q_{LP}(x, \xi) - Q_{LP}(x, \xi')] + [Q_{LP}(x, \xi') - Q(x, \xi')] \\
&\leq \gamma + [Q_{LP}(x, \xi) - Q_{LP}(x, \xi')] + 0 \\
&\leq \gamma + M_{\pi} \left[\|h(\xi) - h(\xi')\|_1 + R\|T(\xi) - T(\xi')\|_1\right].
\end{align}

Since this bound holds for all $x \in X$:
\begin{equation}
R(\xi, \xi') = \sup_{x \in X} [Q(x, \xi) - Q(x, \xi')] \leq M_{\pi} \left[\|h(\xi) - h(\xi')\|_1 + R\|T(\xi) - T(\xi')\|_1\right] + \gamma.
\end{equation}

This suggests the ground cost:
\begin{equation}
c_{MILP}(\xi, \xi') = M_{\pi} \left[\|h(\xi) - h(\xi')\|_1 + R\|T(\xi) - T(\xi')\|_1\right] + \gamma,
\end{equation}
which satisfies regret domination with $\beta = 1$ for the regret $R(\xi, \xi') = \sup_{x \in X} [Q(x, \xi) - Q(x, \xi')]$.
\end{proof}

The bound in Proposition~\ref{prop:milp-integrality-gap} depends on the integrality gap $\gamma$, which can be zero for totally unimodular constraint matrices (e.g., network flow problems).

\begin{example}[Capacitated Network Design with Integer Flows]
Consider a two-stage stochastic network design problem on a directed graph $G = (V, A)$ where
\begin{itemize}
\item first-stage decisions $x \in \{0,1\}^{|A|}$ determine which arcs to open,
\item second-stage decisions route integer commodity flows subject to arc capacities.
\end{itemize}

The second-stage problem for scenario $\xi$ is
\begin{align}
Q(x, \xi) = \min \quad & \sum_{a \in A} \sum_{c \in C} q_{ac} y_{ac} \\
\text{s.t.} \quad & \sum_{a: a^{out}(v)} y_{ac} - \sum_{a: a^{in}(v)} y_{ac} = d_{vc}(\xi) \quad \forall v \in V, c \in C \\
& \sum_{c \in C} y_{ac} \leq u_a x_a \quad \forall a \in A \\
& y_{ac} \in \mathbb{Z}_+ \quad \forall a \in A, c \in C
\end{align}

where
\begin{itemize}
\item $C$ is the set of commodities to be routed,
\item $y_{ac}$ is the integer flow of commodity $c$ on arc $a$,
\item $q_{ac}$ is the unit cost of routing commodity $c$ on arc $a$,
\item $u_a$ is the capacity of arc $a$,
\item $d_{vc}(\xi)$ is the net demand for commodity $c$ at node $v$ under scenario $\xi$ (positive for sinks, negative for sources, zero for transshipment nodes),
\item $a^{out}(v)$ and $a^{in}(v)$ denote the outgoing and incoming arcs at node $v$.
\end{itemize}

Only demands $d_{vc}(\xi)$ vary with scenarios (RHS uncertainty). If $|\pi_{vc}^*| \leq \bar{\pi}_{vc}$ uniformly over all dual optimal solutions, then by Proposition~\ref{prop:milp-integrality-gap} (noting that demands appear linearly in the RHS):

\begin{equation}
|Q(x, \xi) - Q(x, \xi')| \leq \sum_{v \in V} \sum_{c \in C} \bar{\pi}_{vc} |d_{vc}(\xi) - d_{vc}(\xi')| + \gamma,
\end{equation}

where $\gamma$ is the integrality gap. For network flow problems with integer capacities and demands, the constraint matrix is totally unimodular, so $\gamma = 0$. This gives the exact bound:

\begin{equation}
|Q(x, \xi) - Q(x, \xi')| \leq \sum_{v \in V} \sum_{c \in C} \bar{\pi}_{vc} |d_{vc}(\xi) - d_{vc}(\xi')|.
\end{equation}

This suggests using the ground cost $c_{ND}(\xi, \xi') = \sum_{v,c} \bar{\pi}_{vc} |d_{vc}(\xi) - d_{vc}(\xi')|$ for network design problems.
\end{example}

In general, even when the integrality gap is uniformly upper bounded, computing tight integrality gap bounds can be difficult. However, for specific problem structures, we can derive tighter a tighter upper bound on the regret that does not explicitly depend on the integrality gap, as we show next.

\subsubsection{Single-Sourcing Capacitated Facility Location Problem}

While Proposition~\ref{prop:milp-integrality-gap} gives a general bound depending on the integrality gap between the second-stage MILP and its LP relaxation, many MILPs have special structure that enables tighter bounds without explicit integrality gap estimates. We illustrate this through examples. These examples serve as illustration how versatile our approach can be, as building any problem-dependant ground cost that dominates the regret will ensure stability of the value function with respect to the optimal transport cost.

\begin{example}[Tight Bound for Single-Sourcing CFL]
\label{ex:cfl-tight-bound}
Consider the single-sourcing CFL where each customer must be served entirely by one facility:
\begin{equation}
Q(y, \xi) = \min \left\{ \sum_{i,j} c_{ij} w_{ij} \xi_j : \sum_i w_{ij} = 1, \sum_j w_{ij}\xi_j \leq K_i y_i, w_{ij} \in \{0,1\} \right\}.
\end{equation}

For fixed facilities $y$, let $\bar{c}_j = \max_i c_{ij}$. Then:
\begin{equation}
|Q(y, \xi) - Q(y, \xi')| \leq \sum_{j=1}^n |\xi_j - \xi'_j| \cdot \bar{c}_j.
\end{equation}

This motivates the ground cost $c_{CFL}(\xi, \xi') = \sum_j |\xi_j - \xi'_j| \bar{c}_j$.
\end{example}

\begin{proof}
The key is exploiting the single-sourcing structure. Let $w^*(\xi)$ be an optimal assignment for scenario $\xi$. For any scenarios $\xi, \xi'$:
\begin{align}
Q(y,\xi) - Q(y,\xi') &\leq \sum_{i,j} c_{ij} w^*_{ij}(\xi') \xi_j - \sum_{i,j} c_{ij} w^*_{ij}(\xi') \xi'_j \\
&= \sum_j \left(\sum_i c_{ij} w^*_{ij}(\xi')\right) (\xi_j - \xi'_j) \\
&= \sum_j c_{i^*_j(\xi'),j} (\xi_j - \xi'_j)
\end{align}
where $i^*_j(\xi')$ is the facility serving customer $j$ under $\xi'$.

Since this holds for any assignment, including the worst-case:
\begin{equation}
Q(y,\xi) - Q(y,\xi') \leq \sum_j \max_i c_{ij} \cdot \max\{0, \xi_j - \xi'_j\} \leq \sum_j \bar{c}_j |\xi_j - \xi'_j|.
\end{equation}

The bound is tight when all customers are served by their most expensive facilities.
\end{proof}

\begin{remark}[High Capacity Case]
When capacities $K_i$ are sufficiently high (i.e., no capacity constraints are binding), customers can be assigned to their cheapest facilities. In this case, we can obtain a sharper bound using $\underline{c}_j = \min_i c_{ij}$ instead of $\bar{c}_j$:
\begin{equation}
|Q(y, \xi) - Q(y, \xi')| \leq \sum_{j=1}^n |\xi_j - \xi'_j| \cdot \underline{c}_j.
\end{equation}
This reflects that without capacity constraints, each customer will be served by its cheapest available facility.
\end{remark}

Applying Proposition~\ref{prop:milp-integrality-gap} would require bounding dual variables for the assignment constraints and the integrality gap, potentially yielding a looser bound. The direct combinatorial argument exploits the specific structure of single-sourcing constraints to avoid dependence on the integrality gap. 

\subsubsection{Unbounded Integer Knapsack}

Some MILP problems exhibit threshold behavior where small parameter changes can cause large objective changes. Even here, we can find appropriate problem dependant ground costs that dominates the regret without involving the integrality gap. Moreover, the following example illustrates that regret domination for MILP may require using a problem-dependent ground cost that captures step function behavior.

\begin{example}[Unbounded Integer Knapsack as Second Stage]
Consider a simple toy two-stage problem where the first stage involves no decisions (or fixed decisions) and the second-stage problem is an unbounded integer knapsack with uncertain capacity
\begin{equation}
Q(x, \xi) = \max \left\{ \sum_{j=1}^n v_j z_j : \sum_{j=1}^n w_j z_j \leq \xi, z_j \in \mathbb{Z}_+ \right\},
\end{equation}
where $\xi$ is the uncertain knapsack capacity and items can be taken multiple times. Note that this fits the framework of Proposition~\ref{prop:milp-integrality-gap} with $h(\xi) = \xi$, $T(\xi) = 0$ (no technology matrix), and the constraint matrix consisting only of weights $w_j$.

The value function $Q(x,\xi)$ is piecewise constant with discrete jumps at critical thresholds. When $\xi$ increases past a multiple of weight $w_j$, an additional copy of item $j$ becomes feasible, causing the value to jump by $v_j$. This creates a step function (which is in particular not Lipschitz continuous).

When all weights are integers, we can derive a stepwise bound using the greatest common divisor. Set $g = \gcd(w_1, \ldots, w_n)$. The value function can only change at capacities that are multiples of $g$, setting $\rho = \max_j v_j/w_j$ we have
\begin{equation}
Q(\xi) - Q(\xi') \leq \rho \cdot g \cdot \left(\left\lfloor \frac{\max(\xi, \xi')}{g} \right\rfloor - \left\lfloor \frac{\min(\xi, \xi')}{g} \right\rfloor\right)
\end{equation}

This bound is a step function that correctly captures that no value change occurs between multiples of $g$. One could also exploit the dynamic programming structure of the problem and get the linear upper bound
\begin{equation}
Q(\xi) - Q(\xi') \leq \rho \cdot |\xi - \xi'|.
\end{equation}

\emph{Numerical comparison.} Consider items with weights $w = (6, 9, 15)$ and values $v = (30, 36, 45)$, giving value-to-weight ratios $(5, 4, 3)$. Thus $\rho = 5$ and $g = \gcd(6, 9, 15) = 3$.

For $\xi = 14$ and $\xi' = 13$, we have
\begin{itemize}
\item both capacities round down to $12$ (the nearest multiple of $g=3$). At capacity $12$, the optimal solution takes $2$ copies of item $1$, giving $Q(14) = Q(13) = Q(12) = 60$,
\item actual regret is $R(14, 13) = Q(14) - Q(13) = 60 - 60 = 0$,
\item stepwise GCD bound is $\rho \cdot g \cdot (\lfloor 14/3 \rfloor - \lfloor 13/3 \rfloor) = 5 \cdot 3 \cdot (4 - 4) = 0$,
\item and the linear bound is $\rho \cdot |14 - 13| = 5 \cdot 1 = 5$, which is worse than the stepwise GCD bound.
\end{itemize}

This illustrates both the flexibility of our approach to get stability (whatever upper bounds the regret works) and how incorporating problem structure (here, the GCD of weights) can lead to tighter bounds that capture the stepwise nature of the value function.

\end{example}

\section{Conclusion and Future Directions}\label{sec:conclusion}

This paper establishes stability results for two-stage stochastic programs using problem-dependent transport costs, providing theoretical justification for approaches like Bertsimas-Mundru \cite{bertsimas2023} that use general costs rather than distances. Classical stability theory relies on the Wasserstein-Fortet-Mourier duality, which breaks down when Euclidean ground metrics are replaced with problem-dependent costs that are not metrics. We developed a direct approach that bypasses this duality entirely.

Our result is Theorem~\ref{thm:main-stability}, which proves that under regret domination—where the worst-case regret $R(\xi, \xi') = \sup_{x \in X} [Q(x, \xi) - Q(x, \xi')]$ is bounded by $\beta \cdot c(\xi, \xi')$—we obtain stability bounds of the form $|v(P) - v(\nu)| \leq \beta \cdot \max\{\mathcal{T}_c(P, \nu), \mathcal{T}_c(\nu, P)\}$. This result works directly with the primal transport formulation without requiring dual representations.

We presented sufficient conditions for establishing regret domination in Section~\ref{sec:sufficient-conditions}:

\begin{enumerate}
\item \emph{General Programs (Corollary~\ref{prop:romisch-22})}: For arbitrary programs with Lipschitz value functions following R\"omisch's growth conditions, we showed that regret domination holds with $\beta = \frac{L \cdot R \cdot \max\{1, \sup_{\xi \in \Xi} \|\xi\|\}}{\alpha}$ when the ground cost satisfies minimal growth $c(\xi, \xi') \geq \alpha\|\xi - \xi'\|$.

\item \emph{Linear Programs (Proposition~\ref{prop:lp-sensitivity})}: For continuous second-stage problems with fixed cost vector $q$, sensitivity analysis yields sharper bounds with regret controlled by $M_{\pi} \cdot [\|h(\xi) - h(\xi')\| + R\|T(\xi) - T(\xi')\|]$, where $M_{\pi}$ bounds the dual variables over the dual feasible set $\{\pi : W^T\pi \leq q\}$.

\item \emph{Mixed-Integer Programs (Proposition~\ref{prop:milp-integrality-gap})}: For discrete second-stage problems with RHS or technology matrix uncertainty, we showed that regret bounds can be obtained by combining LP relaxation sensitivity with integrality gap estimates. When the constraint matrix is totally unimodular (e.g., network flows), the integrality gap vanishes and we recover exact bounds. For general MILPs, problem-specific analysis often yields tighter bounds by directly exploiting combinatorial structure rather than relying on worst-case integrality gaps.
\end{enumerate}

The treatment of the cases where the second-stage problem is a MILPs is more challenging than its continuous counterpart. While the lack of strong duality prevents direct application of classical techniques, we have shown that
\begin{itemize}
\item general bounds can be obtained via LP relaxation plus integrality gap analysis, providing a systematic framework when problem-specific analysis is difficult,
\item for structured problems (network flows, facility location, knapsack), our flexibility in the ground cost design allows the use of direct combinatorial arguments, which often yield tighter bounds.
\end{itemize}

The regret domination concept likely extend beyond the pure LP and MILP cases. For strongly convex second-stage problems, sensitivity analysis still applies but requires bounding subgradients rather than dual variables, with regret bounds depending on the modulus of strong convexity and Lipschitz constants of the gradients. For risk-aware formulations involving CVaR and other risk measures, regret domination can be established using the dual representation of the risk measure, with constants depending on the risk aversion level and the underlying problem structure.

Problem-dependent costs have shown computational potential \cite{bertsimas2023,keutchayan2023}. In the near future, we intend to provide extensive numerical comparisons between different ground cost choices in the context of scenario reduction. Such future work will also develop an iterative scheme to adaptively compute problem-dependent costs based on current scenario approximations, potentially reducing the computational burden while maintaining solution quality. 

This work provides the theoretical foundation for more effective scenario reduction in stochastic programming. It justifies optimal transport based approach with a problem dependant ground cost instead of a distance. For continuous problems, sensibility analysis yield relevant problem-based ground costs. For discrete problems, we opened new avenues for scenario reduction that respect the combinatorial nature of the decisions.

\section*{Acknowledgments}

This study received funding from the European Union -- Next-GenerationEU -- National Recovery and Resilience Plan (NRRP) –- Mission 4, Component 2, Investment n.~1.1 (call PRIN 2022 D.D. 104 02-02-2022, project title ``Large-scale optimization for sustainable and resilient energy systems'', CUP I53D23002310006) and from PGMO IROE Project N°P-2022-0020 "A Clear Win-Win Case: Interfacing \texttt{SMS++} with PyPSA".

\bibliographystyle{plain}
\bibliography{reports}

\appendix
\
% Include appendices, uncomment 
% \include{stability_proof_appendices}

\end{document}